\documentclass[11pt]{article}
\usepackage{indentfirst}
\usepackage{epsfig}
\usepackage[margin=1in]{geometry}
\usepackage{amssymb, amsmath, mathrsfs, dsfont, amsthm, graphicx, hyperref}

\newtheorem{thm}{Theorem}[section]
\newtheorem{theorem}[thm]{Theorem}
\newtheorem{lemma}[thm]{Lemma}

\newtheorem{proposition}[thm]{Proposition}
\newtheorem{corollary}[thm]{Corollary}

\newtheorem{rem-eg}[thm]{Remark and Example}

\def\be{\begin{equation}}
\def\ee{\end{equation}}
\def\bea{\begin{eqnarray}}
\def\eea{\end{eqnarray}}
\def\bes{\begin{eqnarray*}}
\def\ees{\end{eqnarray*}}
\def\ba{\begin{array}}
\def\ea{\end{array}}

\def\nn{\nonumber}
\def\lb{\label}

\def\pt{\partial}

\def\R{{\bf R}}
\def\C{{\bf C}}

\def\aa{{\alpha}}
\def\bb{{\beta}}

\def\th{{\theta}}

\def\om{{\omega}}

\def\lm{{\lambda}}
\def\Lm{{\Lambda}}

\def\<{{\langle}}
\def\>{{\rangle}}

\def\diag{{\rm diag}}

\def\r{\right}
\def\l{\left}
\def\d{{\mathrm{d}}}
\def\ii{{\mathrm{i}}}
\def\pt{\partial}

\title{Homographic Solutions of the N-body Generalized Lennard-Jones System}
\author{Bowen Liu$^{1}$\thanks{Partially supported by NSFC (No. 11131004, No. 11671215).
           Email: bowen.liu@mail.nankai.edu.cn} \\
$^{1}$ Chern Institute of Mathematics\\Nankai University, Tianjin 300071, China}

\begin{document}

\maketitle

\begin{abstract}
    In this paper, we obtain the existence of non-planar circular homographic solutions and non-circular homographic solutions of the $(2+N)$- and $(3+N)$-body problems of the Lennard-Jones system. 
    These results show the essential difference between the Lennard-Jones potential and the Newton's potential of universal gravitation. 
\end{abstract}

{\bf 2010 MS classification:} 34C25, 70H12, 82C22

{\bf Key words:} Homographic solution, Lennard-Jones potential, non-planar circular solution.

{\bf Running title:} Homographic Solutions of N-body Generalized Lennard-Jones System.

\renewcommand{\theequation}{\thesection.\arabic{equation}}
\renewcommand{\thefigure}{\thesection.\arabic{figure}}

\setcounter{figure}{0}
\setcounter{equation}{0}

\numberwithin{equation}{section}
\section{Introduction}
In this paper, we consider the existence of special homographic solutions of the generalized Lennard-Jones (L-J for short) system. 
The configuration space in $\R^3$ is defined by
$$ X  \;=\; \l\{ q=(q_1, q_2,...,q_N) \in (\R^3)^N| \sum_{k=1}^N m_k q_k = 0, q_k \not= q_j, k\not= j\r\}.  $$
The generalized $N$-body L-J potential is given by  
$$ U(q) \;=\; \sum_{k<j} m_k m_j\left(\frac{1}{|q_k-q_j|^{\bb}}-\frac{1}{|q_k-q_j|^{\aa}} \right).$$
For simplicity, we consider the case of $m_i = m_j = 1$. By the Newton's second law, the motion equations are given by
\be \ddot{q}=- \nabla U(q),\ee
where $\nabla U(q)$ is the gradient of $U(q)$. Then the motion of $q_k$ satisfies
\bea \ddot{q}_k \;= \; \sum_{j\not=k} \left(\frac{\bb}{|q_{kj}|^{\bb+2}}-\frac{\aa}{|q_{kj}|^{\aa+2}} \right)(q_k-q_j),  \lb{1.2} \eea
where  $q_{kj} = q_k-q_j$ and $q_{kj}= - q_{jk}$. Let $q = q(t)$ be a solution of (\ref{1.2}).
It is {\bf flat}, if $q(t)$ is contained
in a plane $\Pi(t)$ in $\R^3$ for any $t\in \R$. It is {\bf planar}, if there is a plane $\Pi$ in $\R^3$ which contains $q(t)$ for all $t\in \R$. 

In the literature, researchers obtained many results of the planar solutions of the $N$-body problems of the L-J potential.
In 2004, when $\aa=6$, $\bb=12$, $a=2$ and $b=1$, Corbera, Llibre and P\'erez-Chavela in 
\cite{CorberaLlibrePerez2004CMDA} and \cite{CorberaLlibrePerez2004RACE} obtained the existence of the constant solutions, the circular solutions and central configurations of the planar $2$- and $3$-body problems of L-J system. In \cite{jones2006n}, Jones studied the central configurations and proved $N$-gon and $(N+1)$-gon are the central configurations of the planar L-J system. In \cite{SbS2010DS}, Sbano and Southall proved the existence of the planar periodic solutions of the $N$-body problem of L-J system when period is large enough by the mountain pass theorem. Readers may refers to \cite{LiuLongZeng2018} and \cite{LlibreLong2015QTDS} for literature review on the Lennard-Jones potential.

In this paper, we focus on the flat solutions and the non-flat solutions. We define that for $1 \leq k\leq N$, $r\in \R^+$ and $\om_k = \frac{2\pi k}{N}$,
\be 
Q_1 = (0,0, r_0), Q_2 = (0, 0,-r_0), Q_3 = (0, 0,0), q_k = \phi(\lm)r_0 (\cos\om_k, \sin\om_k, 0), \lb{1.3}
\ee
where $\phi(\lm) = \sqrt{\lm^2-1}$ and define
\be E(t) = \left(\begin{matrix}
    \cos\om(t) & -\sin\om(t) & 0 \\
    \sin\om(t) & \cos\om(t)  & 0 \\
    0         &       0    & 1
\end{matrix}\right). \lb{1.4}\ee 
If $\om(t) = \om_0 t$ for some constant $\om_0\neq 0$, we write $E(t)$ as $E_{\om_0}(t)$ for simplicity. 

At first, we consider the circular solution $q(t)$ of the $(2+N)$-body problem satisfying 
\be 
q(t) = E_{\om_0}(t)q_0, \lb{1.5}
\ee
where $q_0 = (Q_1, Q_2, q_1, ..., q_N)\in (\R^3)^{2+N}$ as in (\ref{1.3}) and $E_{\om_0}(t)$ is given by (\ref{1.4}).
Note that angular momentum is a constant vector parallel with the $z$-axis, i.e., $\dot{\om} \equiv \om_0$. In this case, $q(t)$ is the solution of (\ref{1.2}) only if there exists a $\lm$ such that following equation holds.
\be
\diag(-\om_0^2, -\om_0^2, 0 )q_0 = \nabla U(q). \lb{1.6}
\ee

\begin{theorem}
    When $0<\aa<\bb$, $N \geq 2$, there exists $\lm_1$ which is given by (iii) of Lemma 2.1 such that for any $\lm\geq \max\{2, \lm_1\}$, $q(t) = E_{\om_0}(t)q_0 \in (\R^3)^{2+N}$satisfying (\ref{1.5}) is a solution of (\ref{1.2}) where $\om_0$ is defined by (\ref{1.4}) and satisfies
    \bea  \om_0  = \pm\l( \frac{\th_{\aa}}{(\sqrt{\lm^2-1} r_0)^{\aa+2}}- \frac{\th_{\bb}}{(\sqrt{\lm^2-1} r_0)^{\bb+2}}
    + \frac{2\aa}{(\lm r_0)^{\aa+2}}-\frac{2\bb}{(\lm r_0)^{\bb+2}}\r)^{\frac{1}{2}}, \lb{1.7}\eea
    $r_0 = \frac{G_1(\lm)}{\sqrt{\lm^2-1}}$, $\th_{\aa}$ and $\th_{\bb}$ are constant given by \eqref{2.2}, and $G_1(\lm)$is given by
    \be G_1(\lm) =\frac{\sqrt{\lm^2-1}}{2\lm}\left(\frac{\bb(2\lm^{\bb +2} + 2^{2+\bb}N)}{\aa(2 \lm^{\aa +2}+ 2^{\aa+2}N)}\right)^{\frac{1}{\bb-\aa}}, \quad {\rm for} \quad \lm \geq 2 .\lb{1.8}\ee
\end{theorem}

By adding another body $Q_3$ at the center of mass, there exists a class of special homographic solutions of the $(3+N)$-body prolem when $N\geq 2$. Suppose that
\be 
q(t) = E_{\om_0}(t)q, \lb{1.9}
\ee
where $q = (Q_1, Q_2, Q_3, q_1, ..., q_N)\in (\R^3)^{3+N}$ is given by (\ref{1.3}) and $E_{\om_0}(t)$ is given by (\ref{1.4}).
\begin{proposition}
    When $0 < \aa < \bb$, there exists a $\lm$ such that for any $\lm\geq \max\{2, \lm_2\}$, $q(t) = E_{\om_0}(t)q$ satisfying (\ref{1.9}) is the homographic solution  of the system (\ref{1.2}), where $r_0 = \frac{G_2(\lm)}{\sqrt{\lm^2-1}}$, $E_{\om_0}(t)$ is defined by (\ref{1.4}), $\om_0$ satisfying 
    \bea \om_0  =  \pm\l(\frac{\th_{\aa}+\aa}{(\sqrt{\lm^2-1} r_0)^{\aa+2}}- \frac{\th_{\bb}+\bb}{(\sqrt{\lm^2-1} r_0)^{\bb+2}}
    + \frac{2\aa}{(\lm r_0)^{\aa+2}}-\frac{2\bb}{(\lm r_0)^{\bb+2}}\r)^{\frac{1}{2}},  \eea
    and $G_2(\lm)$ is given by
    \be G_2(\lm) = \frac{\sqrt{\lm^2 -1}}{2\lm}\left(\frac{\bb(2^{\bb +2}\lm^{\bb +2}+2\lm^{\bb+2}+2 ^{\bb+2}N) }{\aa(2^{\aa +2}\lm^{\aa +2}+2\lm^{\aa+2}+  2^{\aa+2}N)}\right)^{\frac{1}{\bb-\aa}}.  \ee
\end{proposition}

In Theorem 1.3, we prove the existence of the non-circular homographic solutions. Suppose
\be q(t)=r(t)E(t)q_0,\lb{1.12}\ee
where $q_0 = (Q_1, Q_2, q_1, ..., q_N)\in (\R^3)^{2+N}$ is given by (\ref{1.3}), $E(t)$ is defined by (\ref{1.4}) and $r_0 \equiv 1$. When $r\in \R^+$, we define 
\be \Psi(r) = \frac{\bb(N2^{\bb+2}+2\lm^{\bb+2})}{(\bb+1)(2\lm)^{\bb+2} r^{\bb+1}}-\frac{\aa(N2^{\aa+2}+2\lm^{\aa+2})}{(\aa+1)(2\lm)^{\aa+2} r^{\aa+1}}. \lb{1.13}\ee
Define $\bar{r}$ as the only positive root $\Psi'(r) = 0$ by
\be \bar r = \frac{1}{2\lm}\left(\frac{\bb(N2^{\bb+2}+2\lm^{\bb+2})}
{\aa(N2^{\aa+2}+2\lm^{\aa+2})}\right)^{\frac{1}{\bb-\aa}}.\lb{1.14}\ee
\begin{theorem}
    When $0<\aa<\bb$ and $N \geq 2$, there exists $\lm_0$ given by Lemma 2.3 such that for $\lm \geq \lm_0$, $q(t) = r(t)E(t)q_0$ satisfying (\ref{1.12}) is the homographic solution of (\ref{1.2}) where $r(t)$ is the solution of following Hamiltonian system with given Hamiltonian energy 
    \be \left\{\ba{cc}
    &\ddot{r} = -\nabla \Psi(r)= \frac{N\bb}{\lm^{\bb+2} r^{\bb+1}}-\frac{N\aa}{\lm^{\aa+2} r^{\aa+1}} + \frac{2\bb}{2^{\bb+2}r^{\bb+1}}- \frac{2\aa}{2^{\aa+2}r^{\aa+1}},  \\
    &H(r, \dot{r};a,b) = h,  \\
    &r(\tau)=r(0), \quad \dot{r}(\tau)=\dot{r}(0), \ea\right.\ee
    where $h$ satisfies $\Psi(\bar r)< h  < \Psi(\Lm)$, $\om(t)$ is defined by $E(t)$ and satisfying 
    \be \om(t)=\pm \int_{0}^{t} \;\sqrt{\frac{(N-2)\bb}{(\lm r)^{\bb+2}}-\frac{(N-2)\aa}{(\lm r)^{\aa+2}}+ \frac{2\bb}{(2r)^{\bb+2}}- \frac{2\aa}{(2r)^{\aa+2}}-\frac{\th_{\bb}}{(\phi r)^{\bb+2}}+\frac{\th_{\aa}}{(\phi r)^{\aa+2}}} \;\d t,
    \lb{1.16}\ee
    $\Lm$ is given by 
    \be \Lm \equiv \frac{1}{2\lm\phi} \left(\frac{\bb\big((N-2)(2\phi)^{\bb+2} + 2(\lm\phi)^{\bb+2}-\th_{\bb}(2\lm)^{\bb+2}\big)}{\aa\big((N-2)(2\phi)^{\aa+2} + 2(\lm\phi)^{\aa+2}-\th_{\aa}(2\lm)^{\aa+2}\big)}\right)^{\frac{1}{\bb-\aa}} .\lb{1.17}\ee
\end{theorem}

It is well known that in the celestial mechanics, if the configuration is flat, the solution must be planar, i.e., the motion must be in a given fixed plane; if the configuration is 3-dimensional, the corresponding solution must be homothetic. In Theorem 1.1 and Theorem 1.3, if $N = 2$, the homographic solution is flat but it is not planar; if $N \geq 3$, the configuration is 3-dimensional, and the corresponding solution is circular but not homothetic. Theses results show the essiential differences between the Newton's law of universal gravitation and the L-J potential.

\section{The Homographic Solutions}
\setcounter{figure}{0}
\setcounter{equation}{0}
\subsection{The circular homographic solutions of the $(2+N)$-and $(3+N)$-body cases}
In this subsection, we consider (\ref{1.5}). 
The shape of configuration $q$ is uniquely determined by $\lm$. We define 
$\th_{\bb} $ and $\th_{\aa}$ by
\be \th_{\bb} =  \frac{\bb}{2} \sum_{j=1}^{N-1}\frac{1}{|1- e^{\ii\omega_{j}}|^{\bb}}, \quad
\th_{\aa} =  \frac{\aa}{2} \sum_{j=1}^{N-1}\frac{1}{|1- e^{\ii\omega_{j}}|^{\aa}}. \lb{2.2}\ee
where $\ii^2 = -1$.
By direct computations, we have following lemma.

\begin{lemma}
    For any given $0 < \aa < \bb$, $N \geq 2$, $G_1(\lm)$ and $G_2(\lm)$ have following properties.
    
    (i) $G_1(\lm)$ and $G_2(\lm)$ are a monotonically increasing functions in $\lm$ when $\lm \geq 2$;
    
    (ii) $\lim_{\lm\to\infty} G_1(\lm) = \lim_{\lm\to\infty} G_2(\lm) =\infty$;
    
    (iii) there exists a $\lm_1=\lm_1(N) >0 $ such that for any $\lm \geq \lm_1$,
    \be G_1(\lm) \geq \left(\frac{\th_{\bb}}{\th_{\aa}}\right)^{\frac{1}{\bb-\aa}};
    \lb{3.22}\ee
    
    (iv) there exists a $\lm_2 = \lm_2(N) >0 $ such that for any $\lm\geq \lm_2$, we have
    \be G_2(\lm) \geq \max\l\{\left(\frac{\th_{\bb}}{\th_{\aa}}\right)^{\frac{1}{\bb-\aa}},
    \left(\frac{\bb}{\aa}\right)^{\frac{1}{\bb-\aa}}\r\}. \lb{3.46}\ee
\end{lemma}

Then we give the proof of Theorem 1.1.

{\bf Proof of Theorem 1.1.}
    To prove $q = (Q_1, Q_2, q_1, ..., q_N)$ satisfies (\ref{1.6}), we calculate $\frac{\partial U}{\partial Q_j}$ for $j=1, 2$ and $\frac{\partial U}{\partial q_k}$ for $1\leq k \leq N$. For $i, j\in \{1,2\}$, we have 
    \bea
    \frac{\partial U}{\partial Q_j}= \sum_{k= 1}^N \left(\frac{\aa(Q_j -q_k)}{|Q_j -q_{k}|^{\aa+2}}-\frac{\bb(Q_j -q_k)}{|Q_j -q_{k}|^{\bb+2}}\right) + \left( \frac{\aa(Q_j -Q_i)}{|Q_{j}- Q_i|^{\aa+2}}-\frac{\bb(Q_j -Q_i)}{|Q_{j} - Q_i|^{\bb+2}}\right). \lb{2.4}
    \eea
    Since $Q_1 = (0, 0,r_0)$, $Q_2 = (0,0, -r_0)$ and $q_k = \phi(\lm)r_0( \cos\omega_k,\sin\omega_k,0)$, we have that
    \bea \sum_{k= 1}^N \left(\frac{\aa}{|Q_j -q_{k}|^{\aa+2}}-\frac{\bb}{|Q_j -q_{k}|^{\bb+2}}\right)(Q_j -q_k)=  \left(\frac{N\aa}{|\lm r_0|^{\aa+2}}-\frac{N\bb}{|\lm r_0|^{\bb+2}}\right)Q_j, \lb{2.5}\eea
    and
    \bea \left(\frac{\aa}{|Q_{j}- Q_i|^{\aa+2}}-\frac{\bb}{|Q_{j} - Q_i|^{\bb+2}}\right)(Q_j -Q_i) =
    \left(\frac{2\aa}{|2r_0|^{\aa+2}}-\frac{2\bb}{|2r_0|^{\bb+2}}\right)Q_j.
    \lb{2.6}\eea
    Note that (\ref{2.5}) and (\ref{2.6}) yield that for $j = 1$ or $2$,
    \be \frac{\partial U}{\partial Q_j} =   \left(\frac{N\aa}{|\lm r_0|^{\aa+2}} -\frac{N\bb}{|\lm r_0|^{\bb+2}}+ \frac{2\aa}{|2r_0|^{\aa+2}}-\frac{2\bb}{|2r_0|^{\bb+2}}\right)Q_j. \lb{2.7} \ee
    Note that for $j=1,2$,
    \be\diag(-\om_0^2, -\om_0^2, 0)Q_j = 0.\ee
    We have following equality holds.
    \be  \frac{\partial U}{\partial Q_j}  = \frac{N\aa}{|\lm r_0|^{\aa+2}}-\frac{N\bb}{|\lm r_0|^{\bb+2}}+ \frac{2\aa}{|2r_0|^{\aa+2}}-\frac{2\bb}{|2r_0|^{\bb+2}}= 0. \lb{3.33}\ee
    By (\ref{3.33}), we obtain that
    \be r_0 = r(\lm) \equiv \frac{1}{2\lm}\left(\frac{\bb(2\lm^{\bb +2} + 2^{2+\bb}N)}{\aa(2 \lm^{\aa +2}+ 2^{\aa+2}N)}\right)^{\frac{1}{\bb-\aa}}. \lb{3.34}\ee
    Note that $\lm > 0$ is needed up to now.
    By (\ref{1.8}), we have that $r(\lm) = \frac{G_1(\lm)}{\sqrt{\lm^2-1}}$ when $\lm \geq 1$.
    
    Note that $\frac{\partial U}{\partial q_k}$ can be calculated by
    \bea \frac{\partial U}{\partial q_k} = \sum_{j\neq k} \left(\frac{\aa}{|q_{kj}|^{\aa+2}}-\frac{\bb}{|q_{kj}|^{\bb+2}}\right)q_{kj} +\sum_{j= 1}^2 \left(\frac{\aa(q_k -Q_j)}{|q_{k}- Q_j|^{\aa+2}}-\frac{\bb(q_k -Q_j)}{|q_{k} - Q_j|^{\bb+2}}\right).\lb{2.11}
    \eea
    Because following discussion is in $xy$-plane, we identify the $xy$-plane with $\C$. Therefore, $q_k =\phi(\lm)r_0e^{\ii\om_k}$ and $q_{kj}= \phi(\lm)r_0 e^{\ii\om_k}(1- e^{\ii\om_{j-k}})$. Hence, we have that
    \begin{align}
    \sum_{j\neq k} \left(\frac{\aa}{|q_{kj}|^{\aa+2}}-\frac{\bb}{|q_{kj}|^{\bb+2}}\right)q_{kj}
    &=  \sum_{j\neq k} \left(\frac{\aa(1- e^{\ii\om_{j-k}})q_k}{(\phi r_0)^{\aa+2}|1- e^{\ii\om_{j-k}} |^{\aa+2}}-\frac{\bb(1- e^{\ii\om_{j-k}})q_k}{(\phi r_0)^{\bb+2}|1- e^{\ii\om_{j-k}} |^{\bb+2}}\right)\nn\\
    &=\sum_{j=1}^{N-1} \left(\frac{\aa(1- e^{\ii\om_{j}})q_{k}}{(\phi r_0)^{\aa+2}|1- e^{\ii\om_{j}} |^{\aa+2}}-\frac{\bb(1- e^{\ii\om_{j}})q_{k}}{(\phi r_0)^{\bb+2}|1- e^{\ii\om_{j}} |^{\bb+2}}\right). \lb{3.36}\end{align}
    Note that $|1-e^{\ii\om_{j}}| = |1-e^{\ii\om_{N-j}}| = \sqrt{2 - 2\cos\om_j}$ and $(1-e^{\ii\om_{j}})+(1-e^{\ii\om_{N-j}}) = 2 - 2\cos\om_j$. Then we have
    \be \sum_{j=1}^{N-1} \frac{1- e^{\ii\om_{j}}}{|1- e^{\ii\om_{j}} |^{\aa+2}} = \frac{1}{2} \left(\sum_{j=1}^{N-1} \frac{1- e^{\ii\om_{j}}}{|1- e^{\ii\om_{j}} |^{\aa+2}} + \sum_{j=1}^{N-1} \frac{1- e^{\ii\om_{N-j}}}{|1- e^{\ii\om_{N-j}} |^{\aa+2}}\right) = \frac{1}{2} \left(\sum_{j=1}^{N-1} \frac{1}{|1- e^{\ii\om_{j}} |^{\aa}} \right). \ee
    Therefore, we can reduce (\ref{3.36}) to
    \bea \sum_{j\neq k} \left(\frac{\aa}{|q_{kj}|^{\aa+2}}-\frac{\bb}{|q_{kj}|^{\bb+2}}\right)q_{kj}&=&\frac{1}{2}\sum_{j =1}^{N-1} \left(\frac{\aa}{(\phi r_0)^{\aa+2}|1- e^{\ii\om_{j}} |^{\aa}}-\frac{\bb}{(\phi r_0)^{\bb+2}|1- e^{\ii\om_{j}}|^{\bb}}\right)q_{k} \nn\\
    &=&\l(\frac{\th_{\aa}}{(\phi r_0)^{\aa+2}} -\frac{\th_{\bb}}{(\phi r_0)^{\bb+2}}\r)q_{k}, \eea
    where $\th_{\bb}$ and $\th_{\aa} $ are given by (\ref{2.2}).
    The second summation of (\ref{2.11}) can be simplified as
    \bea 
    \sum_{j= 1}^2 \left(\frac{\aa(q_k -Q_j)}{|q_{k}- Q_j|^{\aa+2}}-\frac{\bb(q_k -Q_j)}{|q_{k} - Q_j|^{\bb+2}}\right) 
    &=& \left(\frac{\aa}{(\lm r_0)^{\aa+2}}-\frac{\bb}{(\lm r_0)^{\bb+2}}\right)(q_k -Q_1 + q_k - Q_2) \nn\\
    &=& \left(\frac{2\aa}{(\lm r_0)^{\aa+2}}- \frac{2\bb}{(\lm r_0)^{\bb+2}}\right)q_k,   \eea
    Then (\ref{2.11}) is reduced to
    \be \frac{\partial U}{\partial q_k} =\left( \frac{\th_{\aa}}{(\phi r_0)^{\aa+2}}-\frac{\th_{\bb}}{(\phi r_0)^{\bb+2}} + \frac{2\aa}{(\lm r_0)^{\aa+2}}-\frac{2\bb}{(\lm r_0)^{\bb+2}}\right)q_k. \lb{2.16}\ee
    By $\diag(-\om_0^2, -\om_0^2, 0 )q_k +\frac{\partial U(q)}{\partial q_k} = 0$, we have that
    \be \omega_0^2 = \frac{\th_{\aa}}{(\phi r_0)^{\aa+2}}- \frac{\th_{\bb}}{(\phi r_0)^{\bb+2}} + \frac{2\aa}{(\lm r_0)^{\aa+2}}-\frac{2\bb}{(\lm r_0)^{\bb+2}}.\lb{2.17}\ee
    By $\lm \geq 2$ and (\ref{3.34}),  we have that
    $$\lm r_0 = \frac{1}{2}\left(\frac{\bb(2\lm^{\bb +2} + 2^{2+\bb}N)}{\aa(2 \lm^{\aa +2}+ 2^{\aa+2}N)}\right)^{\frac{1}{\bb-\aa}} \geq \left(\frac{\bb}{\aa}\right)^{\frac{1}{\bb-\aa}}.$$
    This yields that
    \be \frac{2\aa}{(\lm r_0)^{\aa+2}}-\frac{2\bb}{(\lm r_0)^{\bb+2}} \geq 0. \lb{2.18}\ee
    
    By Lemma 2.1, we have that there exists a $\lm_1 = \lm_1(N)$ which depends on $N$ such that for any $\lm \geq \lm_1$, $\phi(\lm) r_0 = G_1(\lm)\geq \left(\frac{\th_{\bb}}{\th_{\aa}}\right)^{\frac{1}{\bb-\aa}}$. This yields
    \be \frac{\th_{\aa}}{(\phi r_0)^{\aa+2}}- \frac{\th_{\bb}}{(\phi r_0)^{\bb+2}} \geq 0. \lb{2.19}\ee
    Hence, by (\ref{2.18}) and (\ref{2.19}), $\om_0$ satisfying (\ref{2.17}) is well-defined for any $\lm \geq \max\{2, \lm_1\}$.
    
    Then $q(t) = E_{\om_0}(t)q$ is a homographic solution of the system (\ref{1.2}) where $q$ satisfies (\ref{1.5}),
    and $\om_0$ is given by (\ref{2.17}). \hfill$\square$

Readers may verify Proposition 1.2 by Lemma 2.1 following the proof of Theorem 1.1. We omit the proof here.

When $N = 2$, the solution is a flat non-planar solution. There is a plane $\Pi(t)$ in $\R^3$ such that $q(t) \in \Pi(t)$  for all $t\in S_{\tau}$ pair-wisely different. In the celestial mechanics, there does not exist any flat non-planar solution.

\begin{corollary} 
    (i) (The (2+2)-body problem) 
    When $0<\aa<\bb$, $N =2$ and $\lm\geq 2$, $q(t)$ satisfying (\ref{1.5}) is the non-planar solution of (\ref{1.2}). Furthermore, the configuration is always a rhombus.
    
    (ii)  (The (3+2)-body problem) When $N =2$, $\lm\geq \max\{2, \lm_2\}$, $q(t)$  satisfying (\ref{1.9}) is a non-planar solution of (\ref{1.2}). Especially, the configuration of $q(t)$ is always a rhombus where $Q_3$ is at the center of this rhombus.
\end{corollary}

\subsection{One General Periodic Homographic Solution}
In this section, we will prove the Theorem 1.3. Before the proof, we need following lemma.
\begin{lemma}
    There exists a $\lm_0$ such that for all $\lm > \lm_0> 0$,  following inequality holds.
     \bea
    && (N2^{\bb+2}+2\lm^{\bb+2})\big((N-2)2^{\aa+2}\phi^{\bb+2} + 2\lm^{\aa+2}\phi^{\bb+2}-\th_{\aa}(2\lm)^{\aa+2}\phi^{\bb-\aa}\big)  \nn\\
    &&<(N2^{\aa+2}+2\lm^{\aa+2})\big((N-2)(2\phi)^{\bb+2} + 2(\lm\phi)^{\bb+2}-\th_{\bb}(2\lm)^{\bb+2}\big).\lb{2.20}
    \eea
\end{lemma}

\begin{proof}
Note that the highest order of $\lm$ on the both sides of (\ref{2.20}) are $4\lm^{\bb+\aa+4} \phi^{\bb+2}$. The second highest order of $\lm$ on the right hand side of (\ref{2.20}) is $N2^{\aa+3}(\lm\phi)^{\bb+2}$ and the second highest order of $\lm$ on the left hand side of (\ref{2.20}) is
$(N-2)2^{\aa+3}\lm^{\bb+2}\phi^{\bb+2}-\th_{\aa}2^{\aa+3}\lm^{\bb+\aa+4}\phi^{\bb-\aa}$.
Note
\be N2^{\aa+3}(\lm\phi)^{\bb+2} > (N-2)2^{\aa+3}(\lm\phi)^{\bb+2}-\th_{\aa}2^{\aa+3}\lm^{\bb+\aa+4}\phi^{\bb-\aa} \ee
is equivalent to
\be 2^{\aa+4}(\lm\phi)^{\bb+2} > -\th_{\aa}2^{\aa+3}\lm^{\bb+\aa+4}\phi^{\bb-\aa},\ee
It holds because the left hand side is positive and the right hand side is negative when $\lm > 0$.  Then we obtain this lemma holds when $\lm_0$ is sufficiently large.
\end{proof}

Now, we give the proof of Theorem 1.3.

{\bf Proof of Theorem 1.3.}
    We plug $q(t)=r(t)E(t)q_0$ defined in \eqref{1.12} into the both sides of (\ref{1.2}) and obtain that the left hand side of (\ref{1.2}) are
    \bea \ddot{Q}_j(t) &=& \left(\ddot{r}(t)E(t) + 2\dot{r}(t)\dot{E}(t)+r(t)\ddot{E}(t)\right)Q_j,\\
    \ddot{q}_k(t) &=& \left(\ddot{r}(t)E(t)+ 2\dot{r}(t)\dot{E}(t)+r(t)\ddot{E}(t)\right)q_k,\eea
    where $1\leq j \leq 2$ and $1\leq k\leq N$ and on the other hand, by (\ref{2.4} -\ref{2.16}), we have
    
    \begin{align}
        \frac{\partial U(q)}{\pt Q_j} &= r(t)E(t)
        \left(\frac{N\aa}{(\lm r)^{\aa+2}} -\frac{N\bb}{(\lm r)^{\bb+2}}+ \frac{2\aa}{(2r)^{\aa+2}}-\frac{2\bb}{(2r(t))^{\bb+2}}\right)Q_j, \lb{3.75}
        \\
        \frac{\partial U(q)}{\pt q_k} &= r(t)E(t) \left( \frac{\th_{\aa}}{(\phi r)^{\aa+2}}-\frac{\th_{\bb}}{(\phi r)^{\bb+2}} + \frac{2\aa}{(\lm r)^{\aa+2}}-\frac{2\bb}{(\lm r)^{\bb+2}}\right)q_k.\lb{3.76}
    \end{align}
    We define $A(t)$ by
    \bea A(t) \equiv r^{-1}(t)E^{-1}(t)(\ddot{r}E(t) + 2\dot{r}\dot{E}(t)+r(t)\ddot{E}(t))= r^{-1}(t)\left(\begin{matrix}
        \ddot{r}-r\dot{\om}^2 & 0                & 0 \\
        0               & \ddot{r}-r\dot{\om}^2  & 0 \\
        0               & 0           & \ddot{r}
    \end{matrix}\right).\eea
    Then $\ddot{Q}_j(t)= r(t)E(t) A(t)Q_{j}= -\frac{\partial U(q(t))}{\partial Q_j}$ and (\ref{3.75}) yield that
    \be \ddot{r} = \frac{N\bb}{\lm^{\bb+2} r^{\bb+1}}-\frac{N\aa}{\lm^{\aa+2} r^{\aa+1}} + \frac{2\bb}{2^{\bb+2}r^{\bb+1}}- \frac{2\aa}{2^{\aa+2}r^{\aa+1}}. \lb{3.78}\ee
    By $\ddot{q}_k =r(t)E(t)A(t)q_k = -\frac{\partial U(q(t))}{q_k}$, (\ref{3.76}) and (\ref{3.78}) yield that
    \be \ddot{r}-r\dot{\om}^2 = \frac{\th_{\bb}}{\phi^{\bb+2} r^{\bb+1}}-\frac{\th_{\aa}}{\phi^{\aa+2} r^{\aa+1}}+\frac{2\bb}{\lm^{\bb+2} r^{\bb+1}}- \frac{2\aa}{\lm^{\aa+2} r^{\aa+1}}. \lb{3.79}\ee
    By (\ref{3.78}) and (\ref{3.79}), we have that
    \be \dot{\om}^2 = \frac{(N-2)\bb}{(\lm r)^{\bb+2}}-\frac{(N-2)\aa}{(\lm r)^{\aa+2}}+ \frac{2\bb}{(2r)^{\bb+2}}- \frac{2\aa}{(2r)^{\aa+2}}-\frac{\th_{\bb}}{(\phi r)^{\bb+2}}+\frac{\th_{\aa}}{(\phi r)^{\aa+2}}.\ee
    If $q(t)$ is the solution of (\ref{1.2}), $\dot{\om}^2 \geq 0$ is a necessary condition, i.e., 
    \be \frac{(N-2)\bb}{(\lm r)^{\bb+2}}-\frac{(N-2)\aa}{(\lm r)^{\aa+2}}+ \frac{2\bb}{(2r)^{\bb+2}}- \frac{2\aa}{(2r)^{\aa+2}}-\frac{\th_{\bb}}{(\phi r)^{\bb+2}}+\frac{\th_{\aa}}{(\phi r)^{\aa+2}} \geq 0 . \lb{3.83}\ee
    Note that (\ref{3.83}) is equivalent to for all $t\in \R$,
    \bea r \leq \Lm(\lm) \equiv \frac{1}{2\lm\phi} \left(\frac{\bb\big((N-2)(2\phi)^{\bb+2} + 2(\lm\phi)^{\bb+2}-\th_{\bb}(2\lm)^{\bb+2}\big)}{\aa\big((N-2)(2\phi)^{\aa+2} + 2(\lm\phi)^{\aa+2}-\th_{\aa}(2\lm)^{\aa+2}\big)}\right)^{\frac{1}{\bb-\aa}}.\eea

    Note that (\ref{3.78}) can be simplified to
    \be \ddot{r} \;=\; -\Psi'(r) \; = \; \frac{\bb(N2^{\bb+2}+2\lm^{\bb+2})}{(2\lm r)^{\bb+2}}-\frac{\aa(N2^{\aa+2}+2\lm^{\aa+2})}{(2\lm r)^{\aa+2}},  \lb{3.81}\ee
    where $\Psi$ is define by (\ref{1.13}).
    Note that $\bar r$ given by (\ref{1.14}) is the only positive root of $\Psi'(\bar r) = 0 $.
    
    As the discussion Theorem 4.1 of \cite{LiuLongZeng2018}, there is an $\tau$-periodic oscillating solution of $r(t)> 0$ for any given Hamiltonian energy $\Psi(\bar{r}) < h <0 $. For any solution $r(t)$ which is not the constant solution, i.e., $r(t) \not\equiv \bar{r}$, we must have $t_1, t_2 \in [0,\tau]$ such that $r(t_1) \leq \bar{r}$ and $r(t_2) \geq \bar{r}$

    If there exists some $\lm_0$ such that $\Lm(\lm_0) > \bar{r}$, then  $r(t)$, which is the periodic solution of (\ref{3.81}), satisfies $\max_{t\in S_{\tau}} r(t) < \Lm(\lm_0)$ for the given Hamiltonian energy $h \in (\Psi(\bar r), \Psi(\Lm(\lm_0)))$. Then for $\lm_0$, $q(t) = r(t)E(\om(t)) q$ is the solution of the system (\ref{1.2}).
    
    Next, we prove the existence of $\lm_0$ which yields for any $\lm \geq \lm_0$, following inequality holds.
    \be \Lm(\lm) >\bar r  . \lb{2.34}\ee
    Note that (\ref{2.34}) holds is equivalent to
    \bea  \frac{1}{2\lm\phi} \left(\frac{\bb\big((N-2)(2\phi)^{\bb+2} + 2(\lm\phi)^{\bb+2}-\th_{\bb}(2\lm)^{\bb+2}\big)}{\aa\big((N-2)(2\phi)^{\aa+2} + 2(\lm\phi)^{\aa+2}-\th_{\aa}(2\lm)^{\aa+2}\big)}\right)^{\frac{1}{\bb-\aa}}>
    \frac{1}{2\lm}\left(\frac{\bb(N2^{\bb+2}+2\lm^{\bb+2})}
    {\aa(N2^{\aa+2}+2\lm^{\aa+2})}\right)^{\frac{1}{\bb-\aa}}.\lb{2.35}\eea
    By direct computations, we have that (\ref{2.35}) holds if and only if  (\ref{2.20}) holds. By lemma 2.3, we have that (\ref{2.35}) holds.   
    
    Therefore, for such $\lm$ which satisfying \eqref{2.34}, $q(t) = r(t)E(t)q_0$ is a solution of the system where $r(t)$ is the classical solution of the ODE (\ref{3.81}) and $\om$ is given by (\ref{1.16}). \hfill$\square$

{\bf Acknowledgment. }{\it This paper is a part of my Ph.D. thesis. I would like to express my sincere thanks to my advisor Professor Yiming Long for his valuable    
    guidance, help, suggestions and encouragements during my study and discussions on this topic.}

\end{document}